\documentclass{elsarticle}
\usepackage{amsmath}
\usepackage{amssymb}
\usepackage[usenames,dvipsnames,svgnames,table]{xcolor}
\usepackage{verse}
\usepackage{todonotes}
\usepackage{mathtools}
\usepackage{url}
\usepackage{hyperref}
\usepackage{amsfonts} 
\usepackage{xspace} 
\usepackage{makeidx}
\usepackage{verbatim}
\usepackage{graphicx}%
 \usepackage[utf8]{inputenc} 
\usepackage[T1]{fontenc} 
\usepackage{enumerate}

 \usepackage{bbm}

\usepackage[small,bf]{caption}

\providecommand{\U}[1]{\protect\rule{.1in}{.1in}}
\usepackage{amsthm}

\newcommand\xqed[1]{%
  \leavevmode\unskip\penalty9999 \hbox{}\nobreak\hfill
  \quad\hbox{#1}}
\newcommand\quadradinho{\xqed{$\triangle$}}

\newenvironment{example}[1][Example]{\addtocounter{theorem}{1} \noindent\textbf{#1 \arabic{theorem}.} }{\qed \\ }

\newcounter{theorem}[section]
\numberwithin{theorem}{section}

\newtheorem{corolario}[theorem]{Corollary}

\newtheorem{definition}[theorem]{Definition}

\newtheorem{lemma}[theorem]{Lemma}

\newtheorem{prop}[theorem]{Proposition}

\newtheorem{teorema}[theorem]{Theorem}

\newcommand{\D}{\partial}
\newcommand\norm[1]{\left\lVert#1\right\rVert}
\newcommand{\pr}[2]{\langle #1,#2\rangle}

\vspace{2cm}

\DeclareMathOperator{\im}{Im}

\DeclareMathOperator{\K}{\mathbb{K}} 

\DeclareMathOperator{\E}{\mathbb{E}}

\DeclareMathOperator{\Prob}{\mathbb{P}}

\usepackage{tikz}
\usetikzlibrary{positioning}
\tikzset{main node/.style={circle,fill=blue!20,draw,minimum size=1cm,inner sep=0pt},}
\usepackage{xcolor}

\usepackage[all,cmtip]{xy}
\usepackage[portuguese,english]{babel}

\date{\today}

\begin{document}

\begin{frontmatter}
\title{Path Cohomology of Locally Finite Digraphs, Hodge's Theorem and the $p$-Lazy Random Walk}
\author[1]{André M. S. Gomes\corref{cor1}}
\ead{andremsg@unicamp.br}
\author[2]{Daniel Miranda}
\ead{daniel.miranda@ufabc.edu.br}
\author[1,3]{Renata Possobon}
\cortext[cor1]{Corresponding author}
\ead{renata.possobon@mis.mpg.de}

\affiliation[1]{{Universidade Estadual de Campinas (UNICAMP)},
addressline={R. Sérgio Buarque de Holanda, 651},
postcode={13083-859},
city={Campinas},
country={Brasil}}

\affiliation[2]{{Universidade Federal do ABC (UFABC)},
addressline={Av. dos Estados 5001},
postcode={09280-560},
city={Santo André},
country={Brasil}}

\affiliation[3]{{Max Planck Institute for Mathematics in the Sciences (MPI MIS)},
addressline={Inselstrasse, 22},
postcode={04103},
city={Leipzig},
country={Germany}}

\begin{abstract}

The study of Markov chains on discrete spaces, such as digraphs, has captivated mathematicians in recent decades due to its interconnectedness with topology, geometry, dynamics, spectral theory, and differential equations.  Furthermore, extensive exploration of these multifaceted relationships has been pursued for their practical utility in diverse fields, including machine learning and image segmentation. In recent times, these interrelations have been generalized to higher dimensions within the framework of finite-dimensional simplicial complexes.

In this paper, we embark on a further extension of these concepts. Initially, we introduce a  cohomology of infinite (though locally finite) digraphs in arbitrary dimensions. Subsequently, in the latter portion of this manuscript, we define a fresh family of Laplace operators and conduct an examination of their spectrum, culminating in the proof of the Hodge Decomposition Theorem within this framework. Finally, we conclude by presenting a Markov chain, the $p$-Lazy Random Walk, whose asymptotic behavior  is intrinsically linked to these cohomologies, while its mixing time is related to the the spectrum of our Laplace operators.

This development opens doors to numerous unexplored questions, particularly regarding potential generalizations of the Ollivier-Ricci curvature to this topology and these Laplacians.




    
	\end{abstract}

 \end{frontmatter}


    \section*{Introduction}

    The relation between stochastic processes and the dynamical, topological and geometrical properties
    of their underlying phase spaces has been extensively investigated in recent decades.
     For instance, diffusion processes associated to Laplace operators, such as the Brownian motion and random walks are also broadly related to studies of stochastic differential equations, 
    the solutions of which are closely associated with the geometry of their phase spaces.
    For a more thorough approach we recommend the books \cite{matthes1986ikeda} and \cite{le2016brownian}.

    A classical example of such powerful relations is the $\frac{1}{2}$-lazy random walk on a graph $G=(V,E)$, in which the walker starts at a given vertex and, at each step it remains in the same with probability $1/2$ and moves to one of its $k$ neighbors with probability $\frac{1}{2k}$. If $\textbf{p}_n^v(w)$ is the probability of the walk starting in the vertex $v$ reaches $w$ after $n$ steps, it is a classical result that the asymptotic behavior of this process revels the topological and algebraic connectivity of the underlying graph, and in particular the dimension of the $0$-th homology.
    
    In \cite{ori}, Parzanchevski and Rosenthal presented a stochastic process on finite and regular simplicial complexes associated to the high-dimensional Laplacian defined in the 1940’s by Eckmann, that generalizes the connections between the topological properties of these complexes and their respective random walks: 
    ``\emph{the asymptotic behavior of the process reflects the existence of a nontrivial $(d-1)$-homology, and its rate of convergence is dictated by the $(d-1)$-dimensional spectral gap}'', \cite{ori}. But due to the fact that they did not have a Hodge decomposition theorem to their Laplacian for infinite complexes, their best results are given for finite complexes.

    In \cite{eidi2023irreducibility}, Eidi and Mukherjee also analyzed stochastic processes on simplicial complexes relating them to the existence of high-dimensional homology and related the irreducibility of these Markov chains on these complexes with their orientability.

    In this paper, based on Parzanchevski's approach, we generalize those notions in a different direction. Instead of studying such relations for simplicial complexes as a generalization of a ``\emph{high-dimensional graph}'' we 
    introduce
        a stochastic process whose asymptotic behavior is related to the cohomology of \emph{path-complexes} on directed graphs, defined in \cite{grig} and \cite{mayer}; which is an cohomology theory for \textit{digraphs} (directed graphs) that is related not only to its vertices and edges, but to its arbitrarily long paths, allowing one to make a precise definition of a high-dimensional hole in a graph, for example. Our approach is valid to infinite digraphs, imposing mild topological and geometrical restrictions. Indeed, the Hodge decomposition theorem is valid to our Laplacians provided that the digraph is locally compact and has its Ricci curvature bounded below ($\Delta d(\cdot,x_0)\geq -C$). 

    In Section \ref{secpath} we generalize the aforementioned path cohomology to locally finite graphs and define notions of adjacency of paths or the graph and of orientability of such paths.
    
    In section \ref{seclaplace} we construct {high-dimensional} Laplace operators, study its corresponding kernels (i.e., harmonic forms) and relate its spectrum to the existence of high homologies.

    In Section \ref{sechodge} we use the Heat Equation and the Heat Kernel to prove the Hodge Decomposition Theorem for this cohomology. An intuitive (and naive) idea arises from statistical mechanics: if the digraph $G$ is the set of possible states to a particle and the initial condition $u_0:G\to\mathbb{R}$ defines the ``\textit{number of particles}'' or the ``\textit{energy}'' of each state at time zero and $u_t$ is its evolution on time $t$ (as the heat dissipates), then this converges to an harmonic form $u_\infty$ that is the ``\textit{thermodynamic equilibrium}''. This equilibrium is the only harmonic form in the path cohomology class of the initial condition $u_0$.




    In Section~\ref{secprocess}, we finally introduce our stochastic process, the $p$-Lazy Random Walk on locally finite digraphs. We investigate its relationship to the spectrum of our Laplacian and the existence of high-dimensional cohomology, also proving the existence of a spectral gap.
    Similar to the cases observed in \cite{ori} and \cite{eidi2023irreducibility}, the $p$-laziness factor appears to prevent periodicity; similarly to the (bipartite) graph case. 

In \cite{joharinad2023mathematical}, Jost and Joharinad demonstrated that the convergence properties of random walks enable the investigation of approximating a Riemannian manifold by graphs. The discrete generators of random walks on graphs and their corresponding Laplace-Beltrami operator have been utilized in Machine Learning, as evidenced by \cite{belkin2003laplacian}, \cite{roweis2000nonlinear}, \cite{tenenbaum2000global}, and \cite{van2008visualizing}. Additionally, \cite{chung2007four} employed random walks to establish Cheeger's inequality, which was further utilized in \cite{shi2000normalized} for studies in Image Segmentation. It remains an open and unexplored question whether the techniques developed herein can be applied in a similar direction.
    





\section{Digraph Weighted Cohomology}\label{secpath}

		A directed graph $G=(V,E)$, or a \textbf{digraph} is a graph in which the edges have a direction. And it is said to be \textbf{locally finite} if every vertex $v\in V$ has a finite number of edges incident to it.
  
  In this section we present the theory of digraph cohomology for a locally finite digraph $G=(V,E)$ with coefficients in field $\K$ based on the path (co)homology presented in \cite{grig}; in which the theory was developed to finite digraphs. 

We call by \textbf{elementary $p$-path on $V$} a sequence $\{i_k\}_{k=0}^p$ of vertices, and denote it by $e_{i_0\cdots i_p}$ or simply by $i_0\cdots i_p$. Given a field $\K$, the space of all formal $\K$-linear combinations of elementary $p$-paths on $V$ is denoted by $\Lambda_p=\Lambda_p(V)$. Its elements are called \textbf{$p$-paths} on $V$.

Observe that the notion of elementary path does not take into account the directed structure of the graph. Indeed, if $a$ and $b$ are vertices of $G$, $e_{ab}$ is an elementary $1$-path independent of the fact that if $(a,b)$ as an edge or not. So, an elementary $p$-path $e_{i_0\cdots i_p}$ is said to be \textbf{allowed} on $G$ if $(i_k,i_{k+1})\in E$, for $k=1, \cdots,{p-1}$. The linear span of the allowed elementary $p$-paths will be denoted by $\tilde{\mathcal{A}_p}$.

 The \textbf{boundary operator} is defined as the linear extension of
\begin{equation}\label{pathboundary}
    \partial e_{i_0\cdots i_p}=\sum_{q=0}^p(-1)^qe_{i_0\cdots \widehat{i_q}\cdots i_p}.
\end{equation}

While our primary focus in this paper is on path cohomology rather than homology, it's important to note that the boundary operator plays a crucial role in defining orientation  and our main notions of adjacency.

Indeed, observe that if $(0,1)$ is an edge, $\partial e_{01}=e_1-e_0$. Of course, in this case, a notion of adjacency in the digraph must relate somehow $e_0$ to $e_1$. But observe that, this boundary also induces an orientation on the vertices, and to capture that we say that $e_1$ is a \textit{neighbor} of $-e_0$ and we say that $e_{01}$ induces a \textit{positive orientation} on $e_1$ and a \textit{negative orientation} on $e_0$. 

We generalize these notions by defining the \textbf{oriented allowed $p$-paths} as the set $G^p_\pm:=\{\pm v:v\in\tilde{\mathcal{A}_p}\}$ and say that $e_{i_0\cdots i_p}$ induces a positive (respectively, negative) orientation on $(-1)^qe_{i_0\cdots \widehat{i_q}\cdots i_p}$ if $q$ is even (resp., odd) and by asking two elementary and oriented allowed $p$-paths to be \textit{neighbors} if they figure in the linear combination that defines the boundary of a same elementary allowed $(p+1)$-path. More precisely:

\begin{definition}
Two oriented allowed elementary $(p-1)$-paths $v$ and $w$ on the digraph $G=(V,E)$ are said to be \textbf{up-neighbors}\index{Neighbor} if there is an allowed elementary $p$-path $e_{i_0\cdots i_p}$ such that there are indexes $r,s\in\{0,\cdots, p\}$ with $r\neq s$ such that
$$v=(-1)^re_{i_0\cdots\widehat{i_r}\cdots i_p}\textrm{ and }w=(-1)^se_{i_0\cdots\widehat{i_s}\cdots i_p}.$$
In this case we denote $v\uparrow w$.
\end{definition}

In this case, we say that $v$ and $w$ are subpaths of $e_{i_0\cdots i_p}$.

Equivalently, we say that any two allowed elementary $(p-1)$-paths, $v$ and $w$ are up-neighbors if there is an allowed elementary $p$-path $\tau$ such that $v$ and $-w$ are subpaths of $\tau$ with the orientations induced by it.

\

\textbf{Remark:} henceforth we will deal only with oriented elementary allowed paths, for this reason we will call them simply by allowed elementary paths, with certain abuse of language.

\

But, in higher orders, one can define to $(p+1)$-paths to be neighbors if they ``\textit{intersect}'' in a $p$-path, as follow.

\begin{definition}
    Two allowed elementary $(p+1)$-paths $v$ and $w$ on the digraph $G=(V,E)$ are said to be \textbf{down-neighbors}\index{Down Neighbor} if there is an allowed elementary $p$-path $e_{i_0\cdots i_p}$ such that there are indexes $r,s\in\{0,\cdots, p\}$ with $r\neq s$ and vertices $k,h\in V$ such that
$$v=(-1)^re_{i_0\cdots{i_{r-1}}ki_{r}\cdots i_p}\textrm{ and }w=(-1)^se_{i_0\cdots{i_{s-1}}h i_{s}\cdots i_p}.$$
In this case we denote $v\downarrow w$.
\end{definition}



Furthermore, we say that a sub-digraph $G'$ of $G$ is an up-$p$-component (respectively, down-$p$-component) of $G$ if for any two elementary $p$-paths $v,v'$ there exists a chain of $p$-paths $v=v_0\uparrow v_1\uparrow\cdots\uparrow v_n= v'$ (resp., $v=v_0\downarrow v_1\downarrow\cdots\downarrow v_n= v'$).

 
If $G$ is an up/down-$p$-component itself we say that $G$ is \textbf{up/down-$p$-connected}.\index{$p$-connected}

\begin{definition}
    A \textbf{disorientation} of the allowed elementary $(p+1)$-paths of $G$ is a choice of orientation $G^{p+1}_+$ for these paths such that no pair of $(p+1)$-paths induces different orientations on any $p$-path.

    And an \textbf{orientation} of the allowed elementary $(p+1)$-paths of $G$ is a choice of orientation for these paths such that no pair of $(p+1)$-paths induces a same orientation on any $p$-path.
\end{definition}

A bipartite graph is a graph with a disorientation on its $1$-paths. And we also highlight that a digraph can admit both a disorientation and an orientation simultaneously.



\begin{definition}
Given an elementary allowed $d$-path $v$ we define as its up and down \textbf{valences}\index{Valence} the maps
$$m_\uparrow(v)=|\{w:\pm w\uparrow v\}| \quad \textrm{and}\quad m_\downarrow(v)=|\{w:\pm w\downarrow v\}|.$$
\end{definition} 

Henceforth \emph{we restrict our analysis to digraphs of bounded valences $m$}. That is, for each $p$ there is an integer $M_p$ such that $m_\uparrow(w),m_\downarrow(w)\leq M_p$ for every elementary $p$-path $w$. 

Furthermore, we say that the \textbf{degree} of a $p$ path $v$ is the number of $p+1$-paths that contains $v$ as a subpath, and we denote it by $\operatorname{deg}(v)$. Furthermore, we say that $G$ is \textbf{$p$-uniform} if every allowed elementary $p$-path is a subpath of at least one allowed elementary $(p+1)$-path. Observe that $m_\uparrow (v)=(p+1)\operatorname{deg}(v)$.

Let us now set the ground to define our cohomology.

	\begin{definition}
		Given a integer $p \geq 0$, a \textbf{$p$-form} on $V$ is a map $f: \Lambda_{p} \to \K$ such that if $v$ is an elementary path, $f(-v)=-f(v)$.
    The $\K$-linear space of all $p$-forms is denoted by $\Lambda^p=\Lambda^p(V)$. 
	\end{definition}

We highlight that the notions of $p$-paths and of $p$-forms were introduced in \cite{grig} without the notion of orientation. But such orientation will be essential to relate our Laplace operators to interesting Markov chains, as in its definition we have an alternating sum.

\begin{example}
    For a given allowed elementary $p$-path $v$ we define its associated Dirac form as
    $$\mathbbm{1}_v(w):=\left\{\begin{array}{ll}
    1, & w=v, \\
    -1, & w=-v,\\
    0, & \textrm{otherwise.}
\end{array}\right.$$
\end{example}

In order to define a cohomological structure, we must define boundary operators.


		\begin{definition}
			The \textbf{boundary operator} is the linear operator $\partial:\Lambda^{p}\to \Lambda^{p-1}$ such that 
			\begin{displaymath}
				(\partial f) (i_0\cdots i_{p-1})=\sum_{k\in V}\sum_{q=0}^{p}(-1)^qf(i_0\cdots i_{q-1} ~ k ~ i_q\cdots i_{p-1})
			\end{displaymath} 
			for every $f \in \Lambda^{p}$. In the same way, we will sometimes denote $\partial$ by $\partial_{(p-1)}$ just to indicate that, after applying $\partial$, a $(p-1)$-form is obtained.
		\end{definition}

		Consider the Hilbert space $l^2(\Lambda_p)$ of the square-integrable $p$-forms, i.e., 
		\begin{displaymath}
			l^2(\Lambda_p):=\left\{ f:\Lambda_p\to\K ~~ : ~~ \sum_{x\in \Lambda_p}f^2(x)<\infty \right\}
		\end{displaymath}
		with  the weighted inner product
		\begin{equation} \label{inner}
			\langle f,g\rangle=\sum_{w\in G^p_\pm}\omega(w)f(w)g(w);
		\end{equation}
    with $\omega:G^p_\pm\to]0,=\infty[$. 

    The standard choices to this weight are 1, $m_\uparrow^{-1}$ and $m_\downarrow^{-1}$. 

    With respect to this inner product explicit computations show that the adjoint of $\partial$ is the \textbf{exterior differential} (or the coboundary) operator given by$$d f(i_0\cdots i_p)=\frac{1}{\omega(i_0\cdots i_p)}\sum_{q=0}^p (-1)^{q}{\omega(i_0\cdots \widehat{i_q}\cdots i_p)}f(i_0\cdots \widehat{i_q}\cdots i_p);$$
in which $\widehat{i_q}$ means the omission of index $i_q$. We highlight that the valence $m$ is bounded -- so it is well defined.
		
		\begin{prop}
			If $f\in l^2(\Lambda_{p-1})$, then $df\in l^2(\Lambda_{p})$.
		\end{prop}
	\begin{proof}
	Observe that
	$$
	\norm{df}^2 = \sum_{i_0,\cdots,i_p\in V}df(i_0\cdots i_p)^2 $$
    $$= \sum_{i_0,\cdots,i_p\in V}\left(\frac{1}{\omega(i_0\cdots i_p)}\sum_{q=0}^p {(-1)^{q}}{\omega(i_0\cdots \widehat{i_q}\cdots i_p)}f(i_0\cdots \widehat{i_q}\cdots i_p)\right)^2$$
	$$= \sum_{i_0,\cdots,i_p\in V}\sum_{q=0}^p\sum_{r=0}^p\frac{(-1)^{q+r}f(i_0\cdots \widehat{i_q}\cdots i_p)f(i_0\cdots \widehat{i_r}\cdots i_p)\omega(i_0\cdots \widehat{i_q}\cdots i_p)\omega(i_0\cdots \widehat{i_r}\cdots i_p)}{\omega(i_0\cdots i_p)^{2}}$$
    $$=\sum_{i_0,\cdots,i_p\in V}\frac{1}{\omega(i_0\cdots i_p)^2}\left(\sum_{q=0}^p\left({f(i_0\cdots \widehat{i_q}\cdots i_p)^2}{\omega(i_0\cdots \widehat{i_q}\cdots i_p)^2} +\right.\right.$$
	$$+ \sum_{r<p}(-1)^{q+r}{f(i_0\cdots \widehat{i_r}\cdots i_q\cdots i_p)f(i_0\cdots i_r\cdots \widehat{i_q}\cdots i_p)}{\omega(i_0\cdots\widehat{i_r}\cdots i_p)\omega(i_0\cdots\widehat{i_q}\cdots i_p)} +$$
	$$\left.\left.+ \sum_{r>p}(-1)^{q+r-1}{f(i_0\cdots \widehat{i_q}\cdots i_r\cdots i_p)f(i_0\cdots i_q\cdots \widehat{i_r}\cdots i_p)}{\omega(i_0\cdots\widehat{i_r}\cdots i_p)\omega(i_0\cdots\widehat{i_q}\cdots i_p)}\right)\right).$$
	
    After changing the indices in latter sum, the last two sums cancel each other. Then
			\begin{displaymath}
				\norm{df}^2=\sum_{i_0,\cdots,i_p\in V}\omega(i_0\cdots i_p)^{-2}\sum_{q=0}^p{f(i_0\cdots \widehat{i_q}\cdots i_p)^2}{\omega(i_0\cdots\widehat{i_q}\cdots i_p)}=\norm{f}^2<\infty.
			\end{displaymath}
		\end{proof}

			A form $f$ is called \textbf{closed} if $f \in \ker d$, and \textbf{exact} if $f \in \im d$.

The following results states that $d$ and $\partial$ are fit to define (co)homological chains.
  
		\begin{lemma}
			$d^2=0$.\label{d2}
		\end{lemma}
		\begin{proof}
		 Of course, this holds to $p=0$ and $p=1$. For $p\geq 1$, note that:
			$$
				(d^2 f) (i_0\cdots i_p) = \frac{1}{\omega(i_0\cdots i_p)}\sum_{q=0}^p (-1)^{q}\omega(i_0\cdots \widehat{i_q}\cdots i_p)df(i_0\cdots \widehat{i_q}\cdots i_p)$$
				$$= \frac{1}{\omega(i_0\cdots i_p)}\sum_{q=0}^p (-1)^{q}\left(\sum_{r=0}^{q-1}(-1)^r\omega(i_0\cdots \widehat{i_r}\cdots \widehat{i_q}\cdots i_p) f(i_0\cdots \widehat{i_r}\cdots \widehat{i_q}\cdots i_p) +\right.$$
    $$\left.\sum_{r=q+1}^{p}(-1)^{r-1}\omega(i_0\cdots \widehat{i_q}\cdots \widehat{i_r}\cdots i_p)f(i_0\cdots \widehat{i_q}\cdots \widehat{i_r}\cdots i_p)\right)$$
			As $(-1)^q(-1)^{r-1}=(-1)^{q+r-1}=-(-1)^{q+r}$, we have that the RHS equals to
   $$\frac{1}{\omega(i_0\cdots i_p)}\left(\sum_{0\leq r< q\leq p}(-1)^{q+r}\omega(i_0\cdots \widehat{i_r}\cdots \widehat{i_q}\cdots i_p)f(i_0\cdots \widehat{i_r}\cdots \widehat{i_q}\cdots i_p)-\right.$$
    $$\left.\sum_{0\leq q< r\leq p}(-1)^{q+r}\omega(i_0\cdots \widehat{i_q}\cdots \widehat{i_r}\cdots i_p)f(i_0\cdots \widehat{i_q}\cdots \widehat{i_r}\cdots i_p)\right).$$
    
	By switching $q$ and $r$, the sums cancel each other out. Then follows $d^2f(i_0\cdots i_p)=0$. 
		\end{proof}


		
	Since $d$ is the adjoint of $\partial$, we have the following result.

		\begin{corolario}
			$\partial^2\equiv 0$.
		\end{corolario}

		
		
		
		 A $p$-form is said to be \textbf{allowed} if it is null on every $\K$-linear combination of not allowed elementary $p$-paths. For a integer $p \geq 0$, we denote $\mathcal{A}^p(V)$ the subspace of $l^2(\Lambda_p)$ of allowed $p$-forms, i. e.,
		\begin{displaymath}
			\mathcal{A}^p(V):=\{f\in l^2(\Lambda_p): f\textrm{ is allowed}\}.
		\end{displaymath}

    \

\textbf{Notation:} Based on our notion of adjacency between $p$-paths we set
$$\sum_{v\uparrow (i_0\cdots i_p)}f(v):=\sum_{k\in V}\sum^p_{q=0}\sum_{r=0}^{q-1}(-1)^{q+r}{f(i_0\cdots \widehat{i_r}\cdots i_{q-1}ki_{q}\cdots i_{p})}+$$
$$+\sum_{k\in V}\sum^p_{q=0}\sum_{r=q}^{p}(-1)^{q+r+1}{f(i_0\cdots i_{q-1}ki_{q}\cdots \widehat{i_r}\cdots i_{p})}$$
and
$$\sum_{v\downarrow(i_0\cdots i_p)}f(v):=\sum_{q=0}^p\sum_{k\in V}\sum_{\substack{r=0 \\ r\neq q}}^{p}(-1)^{q+r}    f(i_0\cdots i_{r-1}k\cdots \widehat{i_q}\cdots i_p).$$

\

		Note that $f$ being allowed does not mean that $df$ or $\partial f$ is also allowed. We will consider the space $\Omega^p(V)$ in which we select only the allowed p-forms $f$ such that $df$ and $\partial f$ are also allowed, i. e.,
		\begin{displaymath}
			\Omega^p(V):=\{f\in \mathcal{A}^p: df\in \mathcal{A}^{p+1}\textrm{ and }\partial f\in\mathcal{A}^{p-1}\}.
		\end{displaymath}
		Therefore, have the chain complexes
		\begin{equation} \label{complex}
   			0 ~ \xrightarrow{\makebox[1cm]{~}} ~  \Omega^0 ~  \xrightarrow{\makebox[1cm]{\scriptsize$d^{(0)}$}} ~ \cdots ~ \xrightarrow{\makebox[1cm]{\scriptsize$d^{(p-2)}$}} ~ \Omega^{p-1} ~ \xrightarrow{\makebox[1cm]{\scriptsize$d^{(p-1)}$}} ~ \Omega^{p} ~ \xrightarrow{\makebox[1cm]{\scriptsize$d^{(0)}$}} ~ \cdots 
		\end{equation}
		and
		\begin{equation}
			0 ~ \xleftarrow{\makebox[1cm]{~}} ~ \Omega^0 ~ \xleftarrow{\makebox[1cm]{\scriptsize$\partial_{(0)}$}} ~ \cdots ~ \xleftarrow{\makebox[1cm]{\scriptsize$\partial_{(p-2)}$}} ~ \Omega^{p-1} ~ \xleftarrow{\makebox[1cm]{\scriptsize$\partial_{(p-1)}$}} ~ \Omega^{p} \xleftarrow{\makebox[1cm]{\scriptsize$\partial_{(p)}$}} ~ \cdots    
		\end{equation}
		
		\

		\begin{definition}
			Over the chain complex \eqref{complex} we define the \textbf{path cohomologies} $H_n$ by
			\begin{displaymath}
				H_n:=\frac{\ker d^{(n)}}{\im d^{(n-1)}}.
			\end{displaymath}
			Which is well defined by Lemma \ref{d2}.
		\end{definition} 


\section{The Discrete Laplacian}\label{seclaplace}

The \textbf{Laplace operator} or the \textbf{Laplacian} is an elliptic operator that plays a fundamental role in analysis, geometry and stochastic calculus. In this section we generalize the usual Laplacian defined for graphs, whose diffusion process is the random walk. To, in the later sections of this work, relate it with the existence of homologies in the subjacent digraph to stochastic processes whose transitions depend on this operator.


\begin{definition}
The \textbf{Laplacian} (or \textbf{Laplace operator}) is the linear operator over $\Omega^p$ given by
$$\Delta:=d\partial+\partial d.$$
\end{definition}

Observe that from the fact that $d$ is adjoint to $\partial$ we conclude that $\Delta$ is self-adjoint.




We now characterize the kernel of this operator to further associate it to the existence of cohomology.

\begin{prop}\label{propkernel}
Let $f$ be a $p$-form. Then, $\Delta f=0$ iff $df=0$ and $\partial f=0$.\label{dd}
\end{prop}
\begin{proof}
Clearly $df=0$ and $\partial f=0$ imply $\Delta \varphi=0$. Now,
$$\langle \Delta f, f \rangle=\langle (d\partial+\partial d)f ,f \rangle = \langle df, df \rangle + \langle \partial f, \partial f \rangle=\norm{df}^2+\norm{\partial f}^2.$$
This way $\Delta f=0$ means that $df=0$ and $\partial f=0$.
\end{proof}

\begin{definition}
A $p$-form $f$ is said to be \textbf{harmonic} if $\Delta f=0.$
And we set the \textbf{space of harmonic $p$-forms} as
$$\mathcal{H}^p:=\{f\in\Omega^p: \Delta f=0\}$$
\end{definition}

\begin{lemma}
$f\in \mathcal{H}^p$ iff $f\in(\im d)^{\perp}\cap\ker d.$
\end{lemma}

\begin{proof}
Suppose That $f \in (\im d)^{\perp}\cap\ker d$. Then $df=0$ and
$$\norm{\partial f}^2=\langle \partial f,\partial f \rangle=\langle f, d\partial f \rangle=0,$$
because $f \in (\im d)^{\perp}$. Then, by Proposition \ref{dd}, $\Delta f=0$.

Now suppose that $\Delta f =0$. Then, again, by Proposition \ref{dd}, $f\in \ker d\cap\ker\partial$. This way
$$\langle f,d g\rangle= \langle \partial f, g\rangle=0.$$
So $f\in(\im d)^{\perp}\cap\ker d.$ 
\end{proof}

\subsection{Up and Down Laplacians and The Spectral Gap}

In this section we split the Laplacian defined above into two differential operators, whose spectrum is related to the topology of the subjacent digraph.

We define then the \textbf{up}\index{Up Laplacian} and the \textbf{down Laplacian}\index{Down Laplacian} respectively by:
$$\Delta^+=\partial d$$
$$\Delta^-=d\partial$$

It is clear that all facts proven to the Laplacian trough this paper holds for the normalized Laplacian $\Delta:=\Delta^++\Delta^-.$

Next proposition establishes explicit formulas for these operators.

\begin{prop}
For any elementary $d$-path $i_0\cdots i_d$ in $G$, and for every $p$-form $f$,
$$\Delta^+ f(i_0\cdots i_p)=f(i_0\cdots i_p)-\sum_{v\uparrow (i_0\cdots i_p)}\omega(v)f(v)$$
and
$$\Delta^- f(i_0\cdots i_p)=\omega(i_0\cdots i_p)^{-1}\sum_{v\downarrow(i_0\cdots i_p)}\omega(v)f(v)$$
\label{laplacianos}
\end{prop}
\begin{proof}
Explicit calculations give:
$$\Delta^+ f(i_0\cdots i_d)=\partial d f(i_0\cdots i_p)=\sum_{k\in V}\sum^p_{q=0}(-1)^qdf(i_0\cdots i_{q-1}ki_{q}\cdots i_{p})=$$
$$f(i_0\cdots i_d)+\sum_{k\in V}\sum^p_{q=0}\sum_{r=0}^{q-1}(-1)^{q+r}{f(i_0\cdots \widehat{i_r}\cdots i_{q-1}ki_{q}\cdots i_{p})}{\omega(i_0\cdots \widehat{i_r}\cdots i_{q-1}ki_{q}\cdots i_{p})}+$$
$$+\sum_{k\in V}\sum^p_{q=0}\sum_{r=q}^{p}(-1)^{q+r+1}{f(i_0\cdots i_{q-1}ki_{q}\cdots \widehat{i_r}\cdots i_{p})}{\omega(i_0\cdots i_{q-1}ki_{q}\cdots \widehat{i_r}\cdots i_{p})}=$$
$$f(i_0\cdots i_p)-\sum_{v\uparrow (i_0\cdots i_d)}\omega(v){f(v)}$$
and
$$\Delta^- f(i_0\cdots i_p)=d\partial f(i_0\cdots i_p)=\frac{1}{\omega(i_0\cdots i_p)}\sum_{q=0}^p(-1)^q \omega(i_0\cdots \widehat{i_q}\cdots i_p)\partial f(i_0\cdots \widehat{i_q}\cdots i_p)=$$
$$\frac{1}{\omega(i_0\cdots i_p)}\sum_{q=0}^p\sum_{k\in V}\sum_{\substack{r=0 \\ r\neq q}}^{p}(-1)^{q+r}\omega(i_0\cdots \widehat{i_q}\cdots i_p)    f(i_0\cdots i_{r-1}k\cdots \widehat{i_q}\cdots i_p).$$
\end{proof}

So we have the following spectral proposition.

\begin{prop}
If $1$ is an upper bound to the weight $\omega$, then the spectrum of $\Delta^+$ is contained in $[0,M+1]$ and the zero is achieved exactly on the closed forms -- in which $M$ is an upper bound to $m_\uparrow$.  \label{spec}
\end{prop}
\begin{proof}
Assume that $\lambda f(v)=(\Delta^+f)(v)$. As $f\in l^2$, $|f(w)|$ is limited and, therefore, achieves a maximum at a point $u$. 

By Proposition \ref{laplacianos}, 
$$\lambda f (u)=(\Delta^+ f)(u)=f(u)-\sum_{v\uparrow u}\omega(v){f(v)}.$$
Thus,
$$|\lambda f(u)|\leq |f(u)|+\sum_{v\uparrow u}|f(v)|\leq (M+1)|f(u)|,$$
since $m_\uparrow\leq M$ and $\omega\leq 1$.

\end{proof}


As a generalization of \cite[Proposition 2.7]{ori}. Next proposition guarantees us sufficient conditions to $M+1$ be achieved as the maximum eigenvalue, given us the \textit{spectral gap} in sense of the difference between the moduli of the two largest eigenvalue of $\Delta^+$.

\begin{prop}
    Fix the weight $\omega=\deg^{-1}$. Suppose that $G$ is up-$(p-1)$-connected and $p$-uniform, then the spectrum of $\Delta^+_p$ lies in $[0,p+1]$ with $0$ being achieved in the closed forms and $p+1$ being achieved iff the digraph is disoriantable, and is achieved on the boundaries of disorientations. 
\end{prop}

\begin{proof}

    The fact that $\operatorname{Sepec}\Delta^+\subset [0,p+1]$ is analogous to the previous proposition.

Now assume that $G$ is up-$(p-1)$-connected and that $G^p_+$ is a disorientation. Define
$$F(\tau)=\left\{\begin{array}{ll}
    1 & \tau \in G^p_+  \\
    -1 & \tau\in G^p_\pm\setminus G^p_+
\end{array}\right.$$
And $f=\partial F$. Thus,
$$f(i_{0\cdots i_{p-1}})=\sum_{k\in V}\sum_{q=0}^{p}(-1)^qF(i_0\cdots i_{q-1} ~ k ~ i_q\cdots i_{p-1})=\deg(i_0\cdots i_{p-1})F(\tau),$$
with $\tau$ being any $p$-path containing $v$ as a subpath.

Moreover, if $w\uparrow v$ there is a $\tau$ such that it contains $v$ as a subpath and $-\tau$ contains $w$ as a subpath, so that
$$\frac{f(v)}{\deg(v)}+\frac{f(w)}{\deg(w)}=F(\tau)+F(-\tau)=0.$$

Thus,
$$\Delta^+f(v)=f(v)-\sum_{w\uparrow v}\frac{f(w)}{\deg(w)}=f(v)-\sum_{w\uparrow v}-\frac{f(v)}{\deg(v)}=f(v)-\sum_{w\uparrow v}F(\tau)=$$
$$=(p+1)f(v).$$

So $f$ is an eigenvector of $\Delta^+$ associated to $p+1$.

On the other hand, assume that $\Delta^+f=(p+1)f$ for some $f\neq 0$. Fix $v$ that maximizes $\frac{|f(v)|}{\deg(v)}$ and normalize $f$ so that $\frac{|f(v)|}{\deg(v)}=1$ and define
$$F=\frac{d f}{p+1},$$
which imposes an disorientation
$$G^p_+=\{\tau\in G^d_\pm : F(\tau)>0\},$$
as we shall see.

Indeed, $f=\frac{\Delta^+f}{p+1}=\frac{\partial d f}{p+1}=\partial F$. Since, by the definition of $\Delta^+$,
$$\deg (w)=|f(w)|=\frac{1}{p}\left|\sum_{v\uparrow w}\frac{f(v)}{\deg(v)}\right|\leq\sum_{v\uparrow w}\frac{|f(v)|}{\deg(v)}\leq \frac{1}{d}\sum_{v\uparrow w}1=\deg(w),$$
so $\frac{|f(v)|}{\deg(v)}=1$ for every $v\uparrow w$ and, consequently, since $G$ is $(p-1)$-connected, $\frac{|f(\cdot)|}{\deg(\cdot)}\equiv 1$. Again from the definition of $\Delta^+$,
$$\frac{f(v)}{\deg(v)}=-\frac{1}{\deg(v)\cdot p}\sum_{w\uparrow v}\frac{f(w)}{\deg(w)}.$$   
Since the r.h.s. an average over terms whose absolute value is that of the l.h.s this gives $\frac{f(v)}{\deg(v)}=\frac{f(w)}{\deg(w)}$ for every $w\uparrow v$. Thus, if $\tau=e_{i_0\cdots i_p}$, WLOG,
$F(\tau)=\frac{1}{p+1}\sum_{q=0}^p\frac{(-1)^qF(e_{i_0\cdots\widehat{i_q}\cdots i_p})}{\deg(e_{i_0\cdots\widehat{i_q}\cdots i_p})}=\frac{f(e_{i_1\cdots i_p})}{\deg(e_{i_1\cdots i_p})}$
which has absolute value constant equals 1. Furthermore, if $\tau,\tau'\in G^p_\pm$ intersect in a $(p-1)$-path $e_{i_0\cdots i_p}$, and induce opposite orientation on it, there are vertices $k,k'\in V$ such that $\tau=(-1)^re_{i_0\cdots i_{r_1}ki_{r}\cdots i_p}$ and $\tau'=(-1)^se_{i_0\cdots i_{s_1}k'i_{s}\cdots i_p}$ with $(-1)^s=-(-1)^r$. So
$$F(\tau)=F((-1)^re_{i_0\cdots i_{r_1}ki_{r}\cdots i_p})=\frac{F(e_{i_0\cdots i_p})}{\deg(e_{i_0\cdots i_p})}=F(-\tau')=-F(\tau')$$

\end{proof}

Following the line of reasoning of Proposition \ref{propkernel}, one shows that $\ker \Delta^+=\ker d$. Since the space of exact closed forms contains the exact forms, these later are called \textbf{trivial zeros} in the spectrum of $\Delta^+$. So, the existence of \textbf{non-trivial} zeros in the spectrum of $\Delta^+$, i.e. closed forms that are not exact, is guaranteed if and only if the digraph has non-trivial homology. So the next definition takes place.

\begin{definition}
    The \textbf{spectral gap} of a locally finite digraph $G$ is 
    $$\lambda(G)=\lambda=\min\operatorname{Spec}(\Delta^+|_{\ker d})=\min\operatorname{Spec}(\Delta|_{\ker d}).$$

    And the \textbf{essential gap} of $G$ is defined by
    $$\tilde{\lambda}(G)=\Tilde{\lambda}=\min\operatorname{Spec}(\Delta^+|_{\im d})=\min\operatorname{Spec}(\Delta|_{\im d}).$$
\end{definition}

The spectral gap is then a quantity that tells us the ``\textit{triviality}'' of the homology of $G$. And $\tilde{\lambda}\neq \lambda$ only if the homology is trivial.


\section{The Hodge Theorem}\label{sechodge}

Our goal in this section is to prove Hodge Decomposition Theorem for this cohomology, which plays a fundamental role in many areas of mathematics by decomposing the space of forms as a direct sum of the kernel of the given Laplace operator and its image.

\begin{teorema}[Hodge Theorem]
 Let $G$ be a locally finite and connected digraph such that there are $x_0\in V$ and $C\geq 0$ with $\Delta d(\cdot,x_0)\geq -C$. Then every path cohomology class has a unique representative that minimizes the norm. This is called the harmonic representative.\label{Hodge}
\end{teorema}

We will deal with this theorem in analogy to what is done in \cite[Chatper 8]{arapura} to the Hodge Theorem on manifolds. It follows -- as will be shown -- from the following theorem; which will be proved in the next section.

\begin{teorema} There are linear operators $\mathcal{H}$ (\textbf{harmonic projection}) and $\mathcal{G}$ (\textbf{Green's operator}) on $\Lambda^p$, for all $p$, that are characterized by the following properties:
\begin{enumerate}
\item $\mathcal{H}(\alpha)$ is harmonic;
\item $\mathcal{G}(\alpha)$ is orthogonal to the space of harmonic forms;
\item $\alpha=\mathcal{H}(\alpha)+\Delta \mathcal{G}(\alpha)$.
\end{enumerate}\label{Hodge2}
\end{teorema}

\begin{corolario}
There is an orthogonal direct sum
$$\Lambda^p=\Delta(\Lambda^p)\oplus \mathcal{H}^p= d\partial(\Lambda^p)\oplus \partial d(\Lambda^p)\oplus \mathcal{H}^p=d(\Lambda^p)\oplus \partial(\Lambda^p)\oplus \mathcal{H}^p.$$
\end{corolario} 

We now prove the Theorem \ref{Hodge} as a consequence of Theorem \ref{Hodge2}.

\begin{proof}
Let $f$ be an exact form. By Theorem \ref{Hodge2} it can be written as
$$f=\mathcal{H}(f)+\Delta \mathcal{G}(f)= \mathcal{H}(f)+d\partial \mathcal{G}(f)+\partial d \mathcal{G}(f).$$
Since $\mathcal{H}(f)$ is harmonic,
$$\langle \partial d \mathcal{G}(f), \mathcal{H}(f)\rangle=\langle d \mathcal{G}(f), d\mathcal{H}(f)\rangle=0. $$
Furthermore,
$$\langle  \partial d \mathcal{G}(f),  d\partial \mathcal{G}(f) \rangle=\langle   d \mathcal{G}(f),  d^2\partial \mathcal{G}(f) \rangle=0.$$
Then
$$\norm{\partial d \mathcal{G}(f)}=\langle \partial d \mathcal{G}(f),f\rangle=\langle d \mathcal{G}(f),df\rangle =0.$$
Thus $f=\mathcal{H}(f)+d\partial \mathcal{G}(f)$ is cohomologous to the harmonic form $\mathcal{H}(f)$.

To prove the uniqueness property, let $f_1$ and $f_2$ be two harmonic forms in the same cohomology class that differ by the exact form $dg$, i.e., $f_1-f_2=dg$. Then
$$\langle dg, f_1-f_2 \rangle= \langle g, \D f_1-\D f_2\rangle=0,$$
given the fact that harmonic forms are in $\ker\D$. Therefore $d g=0$ and $f_1=f_2$.
\end{proof}

\subsection{The Heat Equation}

In this section we will prove the Theorem \ref{Hodge2}. The pictorial idea behind this proof is that a given initial temperature, given by a form $u_0$, should spread through the digraph ``uniformly''; i.e., to $u_0$ should converge to a harmonic form. To do so we shall solve the Cauchy Problem (CP):
$$\left\{\begin{array}{ll}
    \frac{\partial }{\partial t}u+\Delta u & = \quad 0  \\
    u(0,x) & =\quad u_0(x) 
\end{array}\right.$$
on $[0,T[\times V$ with initial condition $u_0\in \Omega^p$.

A map $p:]0,\infty[\times V^{p+1}\times V^{p+1}\to\mathbb{R}$ is said to be fundamental solution to the heat equation,
$$ \frac{\partial }{\partial t}u+\Delta u  =  0 ,$$
if for any bounded initial condition $u_0\in \Omega^p$, the function
$$u(t,x)=\sum_{y\in V}p(t,x,y)u_0(y),\qquad t>0,x\in V$$
is differentiable in $t$, satisfies the heat equation and if for any allowed path $x$,
$$\lim\limits_{t\to 0^+}u(t,x)=u_0(x).$$

We shall, for now on, consider the following generalization of the usual metric on connected digraphs to $p$-connected digraphs. That is, $d(x,x)=0$ and, if $x\neq y$, there is a finite number of $p$-paths $x=x_0\uparrow x_1\uparrow \cdots\uparrow x_k=y\in G^p_\pm$ which connect $x$ and $y$. Then, $d(x,y)$ is the smallest number $k$ of such $p$-paths.

Finally, the following theorem takes place as generalization of \cite[Corollary 4.17]{weber}.

\begin{teorema}\label{6.0.1}
Let $G=(V,E)$ be a locally finite and $p$-connected digraph such that there are $x_0\in V$ and $C\geq 0$ with $\Delta d(\cdot,x_0)\geq -C$. Then the following holds true:
\begin{enumerate}
\item There is a unique fundamental solution $p:]0,\infty[\times V^{p+1}\times V^{p+1}\to\mathbb{R}$ of the heat equation.
\item $G$ is stochastically complete, i.e.
$$\sum_{y\in V}p(t,x,y)=1$$
for any $t>0$ and any allowed path $x$.
\item For every $u_0\in l^1$ and the corresponding bounded solution $u$ of (CP) we have
$$\sum_{x\in V}u(t,x)=\sum_{x\in V}u_0(x)$$
for any $t>0$.
\end{enumerate}
\end{teorema}

\begin{proof}
				First, note that a bounded solution $u$ of (CP) is uniquely determined by $u_0$. In fact, consider $M_1 = \sup \{ |u(t, x)| : t \in (0, T), x \in V \}$ and $M_2 = \sup \{ |u_0(x)| : x \in V \}$. For $r \in \mathbb{N}$, let
				\begin{displaymath}
					v(t, x) = u(t, x) - M_2 - \frac{M_1}{r} ( d(x, x_0) + Ct),
				\end{displaymath}
				and define $\overline{B_r}(x_0) := \{ x \in V : d(x, x_0) \leq r \}$ and $B_r(x_0) := \{ x \in V : d(x, x_0) < r \}$. If $(t, x) \in (\{0\} \times \overline{B_r}) \cup ([0, T) \times \partial B_r)$, then $v(t,x) \leq 0$. On $[0, T) \times B_r$, we have
				\begin{align*}
					\left( \frac{\partial}{\partial t} + \Delta \right) v(t, x) =& \frac{\partial}{\partial t} v(t, x) + \Delta v(t,x) \\
					=& \frac{\partial}{\partial t} u(t, x) -\frac{C M_1}{r} + \Delta u - M_1 \frac{\Delta d(x_0, x)}{r} \\
					= & -\frac{M_1}{r} (\Delta d(x_0, x) +C) \\
					\leq & ~ 0.
				\end{align*}
				From he maximum principle, $v(t, x) \leq 0$ on $[0, T) \times \overline{B_r}$. Then, $u(x, t) \leq M_2 + \frac{M_1}{r} ( d(x, x_0) + Ct)$. Passing it to the limit $r \to \infty$, $u(t, x) \leq M_2$ on $[0, T) \times V$. Using the same argument for $-u$, follows $|u(t, x)| \leq M_2$ for $(x, t) \in [0, T) \times V$. Now, note that if $u_1, u_2$ are bounded solutions with the same initial condition $u_0$, then $u_1 - u_2 = 0$, that is, the bounded solution is uniquely determined by the initial condition.
				
				Item 1 follows directly from this uniqueness. Moreover, that stochastic completeness is equivalent to uniqueness of bounded solutions to the heat equation \cite{Woj07}, then item 2 follows. Finally, note that by items 1 and 2,
				\begin{align*}
					\sum_{x \in V} u(t, x) =& \sum_{x \in V} \sum_{y \in V} p(t, x, y) ~u_0 (y) \\
					= & \sum_{y \in V} u_0 (y) \sum_{x \in V} p(t, x, y) \\
					= & \sum_{y \in V} u_0 (y)
				\end{align*}
				and then item 3 follows.		
			\end{proof}

We will then divide the proof of \ref{Hodge2} in the following propositions.

\begin{prop}\label{7.3.1}
If $u(t,x)$ is a general solution of the heat equation, then $\norm{u(t,x)}^2$ is (nonstrictly) decreasing.
\end{prop}

\begin{proof}
$$\frac{\partial}{\partial t}\norm{u(t,x)}^2=2\langle\frac{\partial}{\partial t} u,u\rangle=-2\langle \Delta u,u\rangle= -2(\norm{du}^ 2+\norm{\partial u}^2)\leq 0.$$
\end{proof}

Let $$T_t(u_0):=\sum_{y\in V}p(t,x,y)u_0(y)$$
with $p$ as in Theorem \ref{6.0.1}. This is, as we have seen, the only solution to (CP).

\begin{prop}\label{semigrupo}
The semigroup property $T_{t_1+t_2}=T_{t_1}T_{t_2}$ holds.
\end{prop}

\begin{proof}
It holds because  $u(t_1+t_2,x)$ can be obtained by solving the heating equation with intial condition $u(t_2,x)$ and then evaluating it at $t_1$.
\end{proof}

\begin{prop}\label{autadj}
$T_t$ is formally self-adjoint.
\end{prop}

\begin{proof} As
$$\frac{\partial}{\partial t}\langle T_t\alpha,T_\tau \beta\rangle=\langle\frac{\partial}{\partial t} T_t\alpha,T_\tau \beta\rangle=-\langle\Delta T_t\alpha,T_\tau \beta\rangle$$
$$=-\langle T_t\alpha,\Delta T_\tau \beta\rangle=\langle T_t\alpha,\frac{\partial}{\partial \tau}T_\tau \beta \rangle=\frac{\partial}{\partial \tau}\langle T_t\alpha, T_\tau \beta \rangle, $$
we can write $\langle T_t\alpha,T_\tau \beta\rangle=g(t+\tau)$. Thus,
$$\langle T_t\alpha,\beta\rangle =g(t+0)=g(0+t)=\langle \alpha, T_t \beta \rangle.$$
\end{proof}

\begin{prop}
$T_t\alpha$ converges to a harmonic form $H(\alpha)$.
\end{prop}

\begin{proof}
We have
$$\norm{T_{t+2h}\alpha-T_t\alpha}^2=\langle T_{t+2h}\alpha-T_t\alpha,T_{t+2h}\alpha-T_t\alpha\rangle=$$ $$\norm{T_{t+2h}\alpha}^2+\norm{T_t\alpha}^2-2\langle T_{t+2h}\alpha,T_t\alpha\rangle.$$
Applying \ref{semigrupo} and then \ref{autadj}, we have
$$\langle T_{t+2h}\alpha,T_t\alpha\rangle=\langle T_{t+h}\alpha,T_{t+h}\alpha\rangle=\norm{T_{t+h}\alpha}^2.$$
Therefore,
$$\norm{T_{t+2h}\alpha-T_t\alpha}^2=(\norm{T_{t+2h}\alpha}-\norm{T_t\alpha})^2-2(\norm{T_{t+h}\alpha}^2-\norm{T_{t+2h\alpha}}\cdot\norm{T_t\alpha}).$$
By (\ref{7.3.1}), $\norm{T_t\alpha}^2$ converges. Thus, $\norm{T_{t+2h}\alpha-T_t\alpha}^2$ can be taken arbitrarily small. As $l^2$ is complete,  $T_t\alpha$ converges, in $l^2$, to a form $H(\alpha)\in l^2$. 

To prove that $\mathcal{H}(\alpha)$ is indeed harmonic, one just need to take $\tau>0$ and note that the relation $T_t\alpha=T_\tau T_{t-\tau}\alpha$ implies, taking the limit $t\to \infty$, that $T_\tau H(\alpha)=\mathcal{H}(\alpha)$. Then $T_t \mathcal{H}(\alpha)$ is constant on $t$. Therefore, $$ 0=\frac{\partial}{\partial t}T_t\mathcal{H}(\alpha)=\frac{\partial}{\partial t}\mathcal{H}(\alpha)=-\Delta \mathcal{H}(\alpha).$$ Thus, $\mathcal{H}(\alpha)$ is indeed harmonic.

Furthermore,
$$\langle \mathcal{H}\alpha,\beta \rangle=\lim\limits_{t\to\infty}\langle T_t\alpha,\beta \rangle=\lim\limits_{t\to\infty}\langle \alpha,T_t\beta \rangle=\langle\alpha, \mathcal{H} \beta\rangle.$$
Then, $H$ is self-adjoint.
\end{proof}

\begin{prop}
$T_t$ is compact for any $t$.
\end{prop}

\begin{proof}
Let  $\{\phi_i\}_i$ be an enumerable orthonormal basis of $l^2(\Lambda^p)$. Then $\phi_i(x)\phi_j(y)$ is also an enumerable orthonormal basis of $l^2(\Lambda^p\times\Lambda^p)$. Define then
$$p_{i,j}=\sum_{x,y\in V}p(t,x,y)\phi_i(x)\phi_j(y),$$
so that
$$p(t,x,y)=\sum_{i,j=1}^\infty p_{i,j}\phi_i(x)\phi_j(y).$$
Now, we define $$p_{n}(x,y)=\sum_{i=1}^n\sum_{j=1}^\infty p_{i,j}\phi_i(x)\phi_j(y),$$
and
$$T_{n}(u(x))=\sum_{y\in V}p_n(x,y)u(y).$$
This way, the image of $T_n$ is a finite dimensional subspace of $l^2(\Lambda^p)$. 

As every operator whose image is finite dimensional is compact, $T_n$ is compact.

Also, we have
$$\norm{(T_t-T_n)u}^2=\norm{\sum_{y\in V}(p(t,x,y)-p_n(x,y))u(y)}^2=$$
$$\sum_{x\in V}\left(\sum_{y\in V}(p(t,x,y)-p_n(x,y))u(y)\right)^2\leq \sum_{x\in V}\left(\sum_{y\in V}(p(t,x,y)-p_n(x,y))^2.\sum_{y\in V}u^2(y)\right)$$
$$=\left(\sum_{x,y\in V}(p(t,x,y)-p_n(x,y))^2\right)\sum_{y\in V}u^2(y)=\sum_{i=n+1}^\infty\sum_{j=1}^\infty |p_{i,j}|^2.\norm{u}^2.$$
As $\norm{p}^2_{l^2}=\sum_{i,j=1}^\infty |p_{i,j}|^2<\infty$, the sum above tends to zero as $n\to \infty$. Thus, $T_n\to T_t$ with respect to operator norm. Therefore we can evoke the theorem below.

Remember that (see \cite{kreyszig}, page 408) if $A:X\to Y$ is a linear map and $A_n\in\mathcal{L}(X,Y)$ is a sequence of compact linear maps such that $A_n\xrightarrow{n\to\infty}A$ with respect to operator norm, then $A$ is compact.
\end{proof}

\begin{prop}
The Green's operator, given by the integral,
$$\mathcal{G}(f)=\int_0^\infty (T_tf-\mathcal{H}f)dt$$
is well defined. Moreover, the following equality takes place
$$\Delta \mathcal{G}(f)=f-\mathcal{H}f.$$

Furthermore, $\mathcal{G}(f)$ is orthogonal to the space of harmonic forms for any $f$.
\end{prop}

\begin{proof}
To show that the Green's operator is well define we must show that $\norm{T_tf(x)-\mathcal{H}f(x)}$ decay rapidly enough.

It is a well known fact that  if $T$ is a compact self adjoint operator on a Hilbert space, $V$, and 
$$m_\uparrow(T):=\sup\{|\pr{Tx}{x}|:x\in V, \norm{x}\leq 1\},$$
then either $m_\uparrow(T)$ or $-m_\uparrow(T)$ is the greatest eigenvalue of $T$ and
$$\norm{T}=|m_\uparrow(T)|$$
holds true.

This way, the maximum of the variational problem:
\noindent\noindent\[m_\uparrow(T_t) := \sup \bigl\{ |\langle T_t f, f \rangle| 
:  \, \|f\| \le 1 \text{ and } Hf=0
 \bigr\},
\]
\noindent{} has solution. We call this maximum by $\beta(t)$  and the correspondent maximum value by $\lambda(t)$.
 
By semigroup property, $T_{2t}\beta=\lambda^2(t)\beta$. On the other hand, $T_{2t}\beta=\lambda(2t)\beta$. Thus,
$\lambda(t+t)=\lambda(t)\cdot \lambda(t)$. Therefore the function $\lambda(t)$  is multiplicative and then $\lambda(t)=e^{-\lambda_1 t}$, with $\lambda_1>0$. Moreover, the minus sign is given by the fact that $T_t\beta$ converges to $\mathcal{H}\beta=0$ as $t\to\infty$.

As $\norm{T_t\alpha - H\alpha}=\norm{T_t(\alpha - \mathcal{H}\alpha)}\leq\norm{T_t\beta}$, we have
$||T_t\alpha - H\alpha||\leq e^{-\lambda_1 t}\norm{\beta}\leq e^{-\lambda_1 t}$. Therefore it decays rapidly enough.

Now,
$$\Delta \mathcal{G}(\alpha)=\int_0^\infty\Delta(T_t\alpha-\mathcal{H}\alpha)dt=\int_0^\infty \Delta T_t\alpha dt=
-\int_0^\infty\frac{\partial T_t\alpha}{\partial t}dt=\alpha-\mathcal{H}(\alpha).$$

At last, let $g$ be a harmonic form. Then
$$\langle \mathcal{G}(\alpha),\beta\rangle = \int_0^\infty\langle(T_t-\mathcal{H})\alpha,\beta\rangle dt=\int_0^\infty \langle\alpha,(T_t-\mathcal{H})\beta \rangle=0.$$
\end{proof}


	\section{The $p$-lazy Random Walk}\label{secprocess}

    In this section we present a stochastic process and show that its asymptotic behavior is highly related to the dimension of the cohomology of $G$. Furthermore, it allows us to give an interesting intuition behind the concept of the spectral and essential gaps.

		\begin{definition}
			For $0<p<1$, the $k$-dimensional \textbf{$p$-lazy random walk}\index{Lazy walk} on $G$ starting on the $k$-elementary path $v$ is the Markov chain on $\Omega_d(G)$ with transition probabilities
$$\Prob(X_{n+1}=v'|X_n=v)=\left\{\begin{array}{l}
p,\quad v'=v\\
\displaystyle\frac{1-p}{m_\uparrow(v)},\quad v'\uparrow v\\
0, \quad \textrm{otherwise.}
\end{array}
\right.$$

\end{definition}

We denote the probability that the random walk which starts at $v$ reaches $w$ at time $n$ by $\textbf{p}_n^v(w)$, which is non-negative. As our explicit formulas to $\Delta^+$ in proposition Proposition \ref{laplacianos} is given in terms of oriented forms (i.e., $f(-v)=-v$), next definition takes place.

\begin{definition}
For $d\geq 2$, the \textbf{expectation process}\index{Expectation process} on $G$ starting at $v$ is the sequence of $d$-forms $\{\mathcal{E}_n^v\}_{n=0}^\infty$ defined by
$$\mathcal{E}_n^v(w)=\textbf{p}_n^v(w)-\textbf{p}_n^v(-w).$$
\end{definition}

The name ``expectation'' is due to the fact that for any $d$-form $f$,
$$\E[f]=\sum_{w\in G_d^\pm}\textbf{p}_n^v(w)f(w)=\sum_{w\in G_d^\pm}\textbf{p}_n^v(w)f(w)=\sum_{w\in G_d^+}\mathcal{E}_n^v(w)f(w).$$

We fix the weight $\omega=m_\uparrow^{-1}$. So the associated \textbf{Transition Operator}\index{Transition Operator} is defined by 
$$A=I-(1-p)\Delta^+.$$

Since $$\textbf{p}_{n+1}^v(w)=p\textbf{p}_{n}^v(w)+\sum_{u\uparrow w}\frac{(1-p)}{m_\uparrow(u)}\textbf{p}_n^v(u),$$ the name Transition Operator is due to the fact that it follows directly from Proposition \ref{laplacianos}, that $\mathcal{E}_n^v=A^n\mathcal{E}_0^v$.

\subsection{Walks and Cohomology}

We now set the ground to relate the asymptotic behavior of this process with the existence of cohomology. But $\lim_n \mathcal{E}_n^v=0$ for every starting point $v$. So we must normalize this chain.

\begin{prop}\label{specA}
The spectrum of $A$ is contained in $[p-M_k(1-p),1]$ and $1$ is achieved on the closed forms.
\end{prop}

\begin{proof}
From the definition of the Transition operator we have that
$$Af(w)=\lambda f(w) \quad \textrm{iff}\quad f(w)-(1-p)\Delta^+ f(w)=\lambda f(w).$$ But the last equality is equivalent to $$\Delta^+ f(w)=\left(\frac{1-\lambda}{1-p}\right)f(w).$$
By Proposition \ref{spec} we have that $$\left(\frac{1-\lambda}{1-p}\right)\in[0,M_k+1].$$ 
Thus $\lambda\in [p-M_k(1-p),1]$.

The fact that the closed forms are stationary, that is, that they are precisely the eigenspace associated to $1$ follows directly to the fact that they are the kernel of $\Delta^+$.
\end{proof}

\begin{prop}
There is a positive integer $K$ s.t. the expectation process satisfies
$$\frac{1}{K}\leq \norm{\mathcal{E}_n^v}\leq \max\left(|p-M(1-p)|^n,1\right) $$
\end{prop}

\begin{proof}
We have that
$$\mathcal{E}_n^v(w)=A^n\mathcal{E}_0^v(w).$$
And $\norm{\mathcal{E}_0^v}=\norm{\mathbbm{1}_v}=\frac{1}{\sqrt{m(v)}}\leq 1$.
Then, by Proposition \ref{specA}, $$\norm{\mathcal{E}_n^v}=\norm{A^n\mathcal{E}_0^v}\leq \norm{A^n}\leq\max\left(|p-M(1-p)|^n,1\right).$$

For the lower bound, let $v=i_0\cdots i_d$, fix $0\leq q\leq d$ and define $$f=d\mathbbm{1}_{i_0\cdots \widehat{i_q}\cdots i_d}=\sum_{w\sim v}m(w)\mathbbm{1}_v.$$
So $f\in \ker \Delta^+$, since $f\in \im d$, and $\norm{f}^2=\sum_{w\sim v}m(w)$. Note that, as $G$ has bounded valence, there is a positive integer $K$ s.t. $\norm{f}^2\leq K$. Since $\Delta^+$ decomposes w.r.t the orthogonal sum $\Omega^{d}=\ker \Delta^+\oplus \im \Delta^+$ so does $A$. Thus,

$$\norm{\mathcal{E}_n^v}=\norm{A^n\mathbbm{1}_v}\geq \norm{\operatorname{proj}_{\ker \Delta^+}(\mathbbm{1}_v)}\geq \left|\left\langle\frac{f}{\norm{f}},\mathbbm{1}_v\right\rangle\right|=\frac{|f(v)|}{\norm{f}m(v)}\geq \frac{1}{K}.$$
\end{proof}

This means that $\norm{\mathcal{E}^v_n}=\Theta\left(|p-M(1-p)|^n\right)$. For this reason we define the \textbf{normalized expectation process} via
$$\widetilde{\mathcal{E}_n^v}=(p-M(1-p))^{-n}\mathcal{E}_n^v.$$

\begin{teorema}

For a $p$-lazy random walk on digraph $G$, we have that for every elementary allowed path $v$, $\mathcal{E}_n^v$ converges to $\mathcal{E}_\infty^v=\operatorname{proj}_{\ker \Delta^+}(\mathcal{E}_0^v)$ and satisfies the following:
\begin{enumerate}
    \item $\mathcal{E}^v_\infty$ is exact if and only if $H_d(G)=0$. Furthermore, if $p\geq \frac{1}{2}$
    \begin{equation}\label{mixingtime}
        d(\mathcal{E}^v_n,\mathcal{E}^v_\infty)=O(1-(1-p)\tilde{\lambda}(G))
    \end{equation}
    \item More generally, $\dim H_d(G)=\dim \operatorname{Span}\{\operatorname{proj}_{\ker\Delta^+}(\mathcal{E}^v_\infty):\, v\textrm{ is an elementary }\textrm{path}\}$.
\end{enumerate}

\end{teorema}

\begin{proof}
Start supposing the the homology of $G$ is trivial. Let us analyze the behavior of $A$ with respect to the orthogonal decomposition $\Omega^d=\ker\Delta^+\oplus\operatorname{Im}\Delta^+$. 

By proposition \ref{specA}, $\operatorname{Spec}(A|_{\operatorname{Im}\Delta^+})\in [p-M(1-p),1[,$
and, trivially, $A|_{\ker \Delta^+}=I|_{\ker \Delta^+}$. This means that $A^n$ converges to the orthogonal projection $\operatorname{proj}_{\ker \Delta^+}$.  Thus
$$\mathcal{E}_\infty^v=\operatorname{proj}_{\ker \Delta^+}(\mathcal{E}_0^v)=\operatorname{proj}_{\ker \Delta^+}(\mathbbm{1}_v).$$
So $\mathcal{E}_\infty^v$ is closed. Therefore, since the homology of $G$ is trivial then it is exact.

To the converse, suppose that $\mathcal{E}_\infty^v$ is exact for every starting point $v$. Then,
\begin{equation}\label{etilde}
    \mathcal{E}_\infty^v=\operatorname{proj}_{\ker \Delta^+}(\mathcal{E}_0^v)=\operatorname{proj}_{\ker \Delta^+}(\mathbbm{1}_v)= \operatorname{proj}_{\operatorname{Im}\partial}(\mathbbm{1}_v)+\operatorname{proj}_{\ker\Delta}(\mathbbm{1}_v),
\end{equation}
which implies that $\operatorname{proj}_{\ker\Delta}(\mathbbm{1}_v)=0$. Since $\{\mathbbm{1}_v\}_v$ is a Schauder basis for $\Omega^d$, it follows that $H_d(G)\simeq \ker \Delta=0$.

Furthermore, from Equation (\ref{etilde})
$$\operatorname{proj}_{ker\Delta^+}(\mathcal{E}_\infty^v)=\operatorname{proj}_{ker\Delta^+}(\operatorname{proj}_{\operatorname{Im}\partial}(\mathbbm{1}_v)+\operatorname{proj}_{\ker\Delta}(\mathbbm{1}_v))=\operatorname{proj}_{\ker\Delta}(\mathbbm{1}_v)).$$

This way,
$$\operatorname{Span}\{\operatorname{proj}_{\ker\Delta^+}(\mathcal{E}_\infty^v)\}_v=\ker\Delta\simeq H_d(G).$$

At last, if $p\geq \frac{1}{2}$, $A$ is positive semidefinite, and we have
$$\norm{A^n-\operatorname{proj}_{\ker\Delta^+}}=\norm{(I-(1-p)\Delta^+)^n|_{\im\Delta^+}}=(1-(1-p)\tilde{\lambda})^n.$$
From which follows Equation (\ref{mixingtime}).
\end{proof}

We close this paper with the consideration that we used the notion of up-adjacency to define our Markov chain. One may define a similar process using the notion of down-adjacency. But it is not trivial to relate this new walk to the spectrum of the Laplacian. If it can be done, one is one step closer to understand the relation of the irreducibility of this new Markov chain with the orientability of the digraph, in analogy to what is done in \cite{eidi2023irreducibility} to simplicial complexes.

\section*{Acknowledgements}

\noindent A.M.S.G and R.P. would like to acknowledge support from the Max Planck Society, Germany, through the award of a Max Planck Partner Group for Geometry and Probability in Dynamical Systems. 

This study was financed in part by the Coordenação de
Aperfeiçoamento de Pessoal de Nível Superior - Brasil (CAPES) - Finance Code 001.

\nocite{*}
\bibliography{referencias}{}
\bibliographystyle{plain}

\end{document}